\documentclass[11pt]{article}

\usepackage{amsthm}
\usepackage{amsmath}
\usepackage{amssymb}
\usepackage[margin=1in]{geometry}
\usepackage[backref=page]{hyperref}
\usepackage{graphicx}
\usepackage{booktabs}
\usepackage{enumitem}

\hypersetup{
    colorlinks=true,     
    linkcolor=blue,      
    citecolor=blue,      
    filecolor=blue,      
    urlcolor=blue        
}

\newtheorem{theorem}{Theorem}
\newtheorem{lemma}{Lemma}

\newtheorem{remark}{Remark}

\newtheorem*{lemma*}{Lemma}

\newcommand{\cA}{\mathcal{A}}

\newcommand{\cV}{\mathcal{V}}

\newcommand{\cW}{\mathcal{W}}

\newcommand{\KKT}{\text{KKT}}

\renewcommand{\S}{\mathbf{S}}

\newcommand{\psd}{\succeq}

\newcommand{\NN}{\mathbb{N}}
\newcommand{\RR}{\mathbb{R}}
\newcommand{\CC}{\mathbb{C}}

\newcommand{\QQ}{\mathbb{Q}}

\DeclareMathOperator{\rankpsd}{rank_{psd}}

\DeclareMathOperator{\conv}{conv}
\DeclareMathOperator{\cl}{\textbf{cl}}
\DeclareMathOperator{\interior}{\textbf{int}}

\DeclareMathOperator{\rank}{rank}

\DeclareMathOperator{\codim}{codim}

\DeclareMathOperator{\elim}{elim}

\DeclareMathOperator*{\argmax}{argmax}

{\begin{small}\begin{list}{}%
         {\setlength{\leftmargin}{1cm}\setlength{\rightmargin}{1cm}}%
         \item[]%
}
{\end{list}\end{small}}

\newcommand{\Vr}{\cV_r}

\title{A lower bound on the positive semidefinite rank of convex bodies}
\author{Hamza Fawzi\footnote{Department of Applied Mathematics and Theoretical Physics, University of Cambridge, UK.} 
\and Mohab {Safey El Din}\footnote{Sorbonne Universit\'es, UPMC Univ Paris 06, CNRS, INRIA Paris Center, 
LIP6, Equipe PolSys, F-75005, Paris, France.}}
\usepackage{color}
\def\modif#1{{{#1}}}

\begin{document}

\maketitle

\begin{abstract}
The positive semidefinite rank of a convex body $C$ is the size of its smallest positive semidefinite formulation. We show that the positive semidefinite rank of any convex body $C$ is at least $\sqrt{\log d}$ where $d$ is the smallest degree of a polynomial that vanishes on the boundary of the polar of $C$. This improves on the existing bound which relies on results from quantifier elimination. Our proof relies on the B\'ezout bound applied to the Karush-Kuhn-Tucker conditions of optimality. We discuss the connection with the algebraic degree of semidefinite programming and show that the bound is tight (up to constant factor) for random spectrahedra of suitable dimension.
\end{abstract}

\section{Introduction}

Semidefinite programming is the problem of optimizing a linear function over a convex set described by a linear matrix inequality:
\[
\max \quad c^T x \quad \text{ s.t. } \quad x \in S
\]
where $S \subseteq \RR^n$ has the form:
\begin{equation}
\label{eq:spec}
S = \{ x \in \RR^n : A_0 + x_1 A_1 + \dots + x_n A_n \succeq 0\}.
\end{equation}
Here $A_0,\dots,A_n$ are real symmetric matrices of size $m\times m$ and the notation $M \succeq 0$ indicates that $M$ is positive semidefinite. A set of the form \eqref{eq:spec} is called a \emph{spectrahedron}.

Given a convex set $C \subseteq \RR^k$, we say that $C$ has a \emph{semidefinite lift} of size $m$ if it can be expressed as
\[
C = \pi(S)
\]
where $S$ is a spectrahedron \eqref{eq:spec} defined using matrices of size $m\times m$ and $\pi$ is any linear map. If $C$ can be expressed in this way, then any linear optimization problem over $C$ can be expressed as a semidefinite program of size $m$. The size of the smallest semidefinite lift of $C$ is called the \emph{positive semidefinite rank} of $C$ \cite{gouveia2011lifts,psdranksurvey}.

The purpose of this paper is to give a general lower bound on the
positive semidefinite rank of convex bodies. \modif{Here, by a convex
  body we mean a closed convex set such that the origin lies in the
  interior of $C$.} For the statement of our main theorem, we need the
notion of \emph{polar} of a convex body $C$, defined as follows:
\[
C^o = \left\{ c \in \RR^k : \langle c,x \rangle \leq 1 \; \forall x \in C\right\}.
\]
The polar of a convex body is another full-dimensional closed convex
set that is bounded and contains the origin \cite[Theorem
14.6]{rockafellar1997convex}. Throughout the paper, we use $\log$ for
the logarithm base 2. \modif{We can now state the first main result of
  this article.}
\begin{theorem}
\label{thm:mainlbpsdrank}
Let $C$ be a convex body and let $C^o$ be its polar. Let $d$ be the smallest degree of a polynomial with real coefficients that vanishes on the boundary of $C^o$. Then $\rankpsd(C) \geq \sqrt{\log d}$.
\end{theorem}
We also show that this bound is tight in general (up to multiplicative factors):
\begin{theorem}
\label{thm:tightness}
There exist convex bodies $C$ where the degree $d$ of the algebraic boundary of $C^o$ can be made arbitrary large and where $\rankpsd(C) \leq \sqrt{20 \log d}$. 
\end{theorem}

When $C$ is a polytope, the degree $d$ of the algebraic boundary of $C^o$ is nothing but the number of vertices of $C$. Theorem \ref{thm:mainlbpsdrank} can thus be compared to the well-known lower bound of Goemans \cite{goemans2009smallest} on the size of linear programming lifts. The \emph{linear programming extension complexity} of a polytope $P$ is the smallest $f$ such that $P$ can be written as the linear projection of a polytope with $f$ facets.
\begin{theorem}[{Goemans \cite[Theorem 1]{goemans2009smallest}}]
\label{thm:goemans}
Assume $P$ is a polytope with $d$ vertices. Then the linear programming extension complexity of $P$ is at least $\log d$.
\end{theorem}
\begin{proof}
The proof is elementary so we include it for completeness. Assume $P = \pi(Q)$ where $Q$ is a polytope with $f$ facets. The pre-image by $\pi$ of any vertex of $P$ is a face of $Q$. Since $Q$ has at most $2^f$ faces it follows that $f \geq \log d$.
\end{proof}

For functions $f,g:\NN\rightarrow \RR$ we say that
$f(n) \in \Omega(g(n))$ if there exists a constant $K > 0$ such that
$f(n) \geq K\cdot g(n)$ for all large enough $n$.

The only previous lower bound on the positive semidefinite rank that
applies to arbitrary convex bodies that we are aware of is the bound
proved in \cite[Proposition 6]{gouveia2011lifts} which says
that\footnote{In the bound shown in \cite[Proposition
  6]{gouveia2011lifts}, $d$ is the degree of the algebraic boundary of
  $C$. However since $\rankpsd(C) = \rankpsd(C^o)$ it can also be
  taken to be that of $C^o$ in the statement of the lower bound.}
$\rankpsd(C) \geq \Omega\left(\sqrt{\frac{\log d}{n \log \log
      (d/n)}}\right)$ where $n$ is the dimension of $C$. This bound
was obtained using results from quantifier elimination theory and one
drawback is that it involves constants that are unknown or difficult
to estimate. Our lower bound of Theorem \ref{thm:mainlbpsdrank}
improves on this existing bound and also has the advantage of being
explicit.

\paragraph{Main ideas.} The main idea behind the proof of Theorem \ref{thm:mainlbpsdrank} is simple. Given a convex body $C$, we exhibit a system of polynomial equations that vanishes on the boundary of $C^o$. This system of polynomial equations is nothing but the Karush-Kuhn-Tucker (KKT) system, after discarding the inequality constraints to get an algebraic variety. Applying the B\'ezout theorem on the KKT system gives us an upper bound on the degree of this variety and yields the stated lower bound. To prove Theorem \ref{thm:tightness} about tightness of the bound we appeal to existing works \cite{nie2010algebraic} where the degree of the KKT system was explicitly computed, under certain genericity assumptions.
The convex bodies of Theorem \ref{thm:tightness} are in fact random spectrahedra (i.e., spectrahedra defined using random matrices $A_0,\dots,A_n$) of appropriate dimension, where the formulas for the algebraic degree of semidefinite programming \cite{von2009general} allow us to lower bound the degree of the algebraic boundary of their polars.
We would like to point out that many of the ideas involved in the proofs of Theorems \ref{thm:mainlbpsdrank} and \ref{thm:tightness} appear in some form or another in  \cite{rostalskidualitiescvxalg,nie2010algebraic,sinn2015generic}. For example a study of the algebraic boundary of polars of spectrahedra appears in \cite[Section 5.5]{rostalskidualitiescvxalg}. However it seems that the connection with the positive semidefinite rank was not made explicit before. The focus in these previous works seemed to be on getting \emph{exact} values for the degrees, at the price of genericity assumptions. In the present work our aim was on getting bounds (tight up to constant factors) but valid without any genericity assumption.

\paragraph{Notations.} The (topological) boundary of a set $C \subseteq \RR^n$ is denoted $\partial C$ and defined as $\partial C = \cl(C) \setminus \interior(C)$ where $\cl$ and $\interior$ denote closure and interior respectively. The algebraic boundary of $C \subseteq \RR^n$ denoted $\partial_a C$ is the smallest affine algebraic variety in $\CC^n$ that contains $\partial C$. 
We denote by $\S^m$ the space of $m\times m$ real symmetric matrices. This is a real vector space of dimension
\[
t_m := \binom{m+1}{2}.
\]
We also denote by $\S^m(\CC)$ the space of $m\times m$ symmetric matrices with complex entries.

\section{Proof of Theorem \ref{thm:mainlbpsdrank}}
\label{sec:proof}

In this section we prove Theorem \ref{thm:mainlbpsdrank}. To do so we will exhibit polynomial equations that vanish on the boundary of polars of spectrahedra and their shadows. These equations are nothing but the KKT conditions of optimality. Applying the B\'ezout bound will yield Theorem \ref{thm:mainlbpsdrank}.

\paragraph{KKT equations.} Let $A_0,\dots,A_n \in \S^m$ and define
\[
\cA(x) := x_1 A_1 + \dots + x_n A_n.
\]
Consider the linear optimization problem
\begin{equation}
\label{eq:optcS}
\max \quad c^T x \quad \text{ s.t. } \quad A_0 + \cA(x) \succeq 0
\end{equation}
and assume that the feasible set
\[
S = \{ x \in \RR^n : A_0 + \cA(x) \succeq 0\}
\]
contains 0 in its interior. In this case we know that any optimal point $x$ of \eqref{eq:optcS} must satisfy the following KKT conditions:
\begin{equation}
\label{eq:KKT}
\exists X, Z \in \S^m \quad : \quad X = A_0 + \cA(x), \quad \cA^*(Z) + c = 0, \quad XZ = 0, \quad X \succeq 0, \quad Z \succeq 0
\end{equation}
\modif{where the variable $Z$ plays the role of \emph{dual multiplier}
  and
  $\cA^*(Z) = ({\sf Trace}(A_1\ Z), \ldots, {\sf Trace}(A_n\ Z))$}.
Conditions \eqref{eq:KKT} consist of equality conditions as well as
inequality conditions. If we disregard the inequality conditions we
get a system of polynomial equations in
$(x,X,Z) \in \RR^n \times \S^m \times \S^m$ which we denote by
$\KKT(c)$:\footnote{We note here that there are multiple ways of
  writing the SDP complementarity conditions in general, and these can
  lead to differences in the context of algorithms for SDP, see e.g.,
  the discussion in \cite[Section 6.5.4]{ben2001lectures}. For our
  purposes, the main property that we will need of the system
  $\KKT(c)$ is that it has a finite number of solutions generically
  (Lemma \ref{lemma:finitenessKKT}).}
\begin{equation}
\label{eq:KKTeq}
\KKT(c) \quad : \quad X = A_0 + \cA(x), \quad \cA^*(Z) + c = 0, \quad XZ = 0.
\end{equation}
This system has $n+2t_m$ unknowns and consists of $n+t_m+m^2$ equations. A crucial fact about this system is that it has a finite number of solutions, assuming the parameters $A_0,A_1,\ldots,A_n$ and $c$ are generic (we come back to the genericity assumption after the statement of the result; some form of genericity is needed for the statement to be true). It is the number of solutions to the KKT system that will give an upper bound on the degree of the algebraic boundary of the polar as we will see later.
\begin{lemma}[Finiteness of KKT solutions]
\label{lemma:finitenessKKT}
For generic $A_0,A_1,\ldots,A_n$ and $c$, the KKT system of polynomial equations \eqref{eq:KKTeq} has a finite number of complex solutions $(x,X,Z) \in \CC^n \times \S^m(\CC) \times \S^m(\CC)$. Furthermore the number of such solutions is at most $2^{m^2}$.
\end{lemma}
\begin{proof}
That the KKT system has a finite number of solutions generically was proved in \cite[Theorem 7]{nie2010algebraic}. We include a sketch of proof for completeness which is simply a dimension count argument. There are three equations in \eqref{eq:KKTeq}:
\begin{itemize}
\item The equation $X = A_0 + \cA(x)$ is linear and defines an affine subspace of codimension $t_m$ (we assume that $\cA$ is injective).
\item The equation $\cA^*(Z) + c = 0$ is also linear and defines an affine subspace of codimension $n$.
\item Finally the equations $XZ = 0$ can be shown to define a variety of codimension $t_m$ (see e.g., \cite[Proof of Theorem 7]{nie2010algebraic}).
\end{itemize}
If $A_0,A_1,\ldots,A_n$ and $c$ are generic, a Bertini-Sard type theorem tells us that the intersection of these three varieties will have codimension equal to the sum of the codimensions, i.e., $t_m + n + t_m = 2t_m + n$ which is the dimension of the ambient space. In other words the variety defined by \eqref{eq:KKTeq} is zero-dimensional, i.e., there are a finite number of solutions.

B\'ezout bound tells us that the number of solutions is at most the product of the degrees of the polynomial equations that form the system \eqref{eq:KKTeq}, which in this case is $2^{m^2}$.
\end{proof}

\begin{remark}[Genericity assumption of Lemma \ref{lemma:finitenessKKT}]
An assumption of genericity is necessary in general to guarantee that the system \eqref{eq:KKTeq} has a finite number of solutions. This is to rule out situations where the optimization problem \eqref{eq:optcS} has an infinite number of solutions (a positive dimensional face of $S$) or when there are an infinite number of dual multipliers. In Lemma \ref{lemma:finitenessKKT} we assumed all the parameters $A_0,\ldots,A_n,c$ generic to be able to apply a standard Bertini-Sard type theorem. We think however it may be possible to remove some of the genericity assumptions (e.g., just to assume genericity on $A_0$ and $c$) but we did not pursue this further here as the current statement of the lemma will be sufficient for our purposes.
\end{remark}

The next lemma shows that the number of solutions to the KKT system is intimately tied to the degree of the algebraic boundary of the polar $S^o$.

\begin{lemma}
\label{lemma:polarspecdeg}
Consider a spectrahedron $S = \{x \in \RR^n : A_0 + x_1 A_1 + \dots + x_n A_n \succeq 0\}$ where $A_0,\ldots,A_n \in \S^m$ and assume that $0 \in \interior S$. Let $S^o$ be its polar defined by 
\[
S^o = \left\{ c \in \RR^n : \langle c,x \rangle \leq 1 \; \forall x \in S \right\}.
\]
Then there is a polynomial of degree at most $2^{m^2}$ \modif{with real coefficients} that vanishes on the boundary of $S^o$.
\end{lemma}
\begin{proof}

The points on the boundary of $S^o$ are exactly those $c$ such that $\max_{x \in S} c^T x = 1$.
Consider the system of polynomial equations obtained by adding the equation $c^Tx = 1$ to the KKT system:
\begin{equation}
\label{eq:KKT2}
\KKT \; : \;
\begin{cases}
X = A_0 + \cA(x), \quad \cA^*(Z) + c = 0, \quad XZ = 0\\
c^T x = 1.
\end{cases}
\end{equation}
We think of \eqref{eq:KKT2} as a system of equations on the variables $(x,X,Z,c)$. If we eliminate the variables $x,X,Z$ we get an algebraic variety $\cV \subset \CC^n$ in the variables $c$:
\begin{equation}
\label{eq:defV}
\cV = \elim_{c}({\sf Sols}(\KKT)).
\end{equation}
By construction this variety \modif{contains} the boundary of $S^o$, i.e., $\partial S^o \subseteq \cV \cap \RR^n$. To bound the degree of $\partial_a S^o$ it thus suffices to count the number of intersections of $\cV$ with a generic line, since $\partial_a S^o \subseteq \cV$ and $\partial_a S^o$ is a hypersurface \cite[Corollary 2.8]{sinn2015algebraic}. We will do this first in the case where $A_0,\ldots,A_n$ are generic. 
  Let $\modif{c_0} \in \CC^n$ generic and consider the line $\{ \lambda \modif{c_0} : \lambda \in \CC\}$. Since $\cV$ was defined by eliminating variables $x,X,Z$ from \eqref{eq:KKT2}, we know that $\lambda \modif{c_0} \in \cV$ if and only if there exist $(x,X,Z)$ in the solution set of $\KKT(\lambda \modif{c_0})$ and $\lambda \modif{c_0}^T x = 1$. By looking at the equations defining $\KKT(\lambda \modif{c_0})$ this \modif{implies} that $(x,X,(\modif{c_0}^T x)Z)$ is in the solution set of $\KKT(\modif{c_0})$. Thus the number of intersection points is at most the cardinality of the solution set of $\KKT(\modif{c_0})$, i.e., $2^{m^2}$. We have thus shown that $\partial_a S^o$ is a hypersurface of degree at most $2^{m^2}$.

  It thus remains to treat the case where $A_0,A_1,\ldots,A_n$ in the definition of $S$ are not generic. This can be done by using a simple perturbation argument.
  \modif{Let $N$ be the total number of the entries in $n+1$ symmetric matrices. Hence, the sequence of matrices $A_0, \ldots, A_n$ represents a point $\mathbf{A}$ in $\RR^N$. For any $k \in \mathbb{N}\setminus\{0\}$, there exists a point $\mathbf{A}_k$ in $\RR^N$ in the ball centered at $\mathbf{A}$ of radius $1/k$ which is generic and represents a sequence of symmetric matrices $A_{0,k}, \ldots, A_{n,k}$. Since, by assumption $0 \in \interior S$, $A_0$ is positive definite, one can assume w.l.o.g. that $A_{0,k}$ is positive definite. Hence the spectrahedra $S_k$ defined by $A_{0,k}, \ldots, A_{n,k}$ are generic, non-empty and such that $0 \in \interior S_k$. Hence, one can apply to them the above paragraph.}
  
  Now, let $(p_k)$ be a sequence of polynomials of degree at most $2^{m^2}$ that vanish on the boundary of $(S_k)$. We can rescale each $p_k$ to be unit-normed and we can thus assume that $(p_k)$ has a convergent subsequence that converges to some polynomial $p$. Clearly the degree of $p$ is at most $2^{m^2}$. Finally it is easy to verify that $p$ vanishes on the boundary of $S^o$.
\end{proof}

We are now in position to prove Theorem \ref{thm:mainlbpsdrank} on the lower bound for the positive semidefinite rank. The main idea is that if $C = \pi(S)$ where $S$ is a spectrahedron, then by duality $C^o$ is the intersection of $S^o$ with an affine subspace and thus the algebraic boundary of $C^o$ has degree at most that of $S^o$.

\begin{proof}[Proof of Theorem \ref{thm:mainlbpsdrank}]
Assume $C$ is a convex body that can be written as $C = \pi(S)$ where $S$ is a spectrahedron defined using an $m\times m$ linear matrix inequality and $\pi$ a linear map. We can assume that $S$ has nonempty interior, and furthermore that $0 \in \interior(S)$ since $0 \in \interior(C)$.
We are going to exhibit a polynomial of degree at most $2^{m^2}$ that vanishes on the boundary of $C^o$. Let $p$ be a polynomial of degree at most $2^{m^2}$ that vanishes on the boundary of $S^o$. Then we claim that the polynomial $q = p \circ \pi^*$ (where $\pi^*$ is the adjoint of $\pi$), which has degree at most $2^{m^2}$ vanishes on the boundary of $C^o$. Indeed if $y$ is on the boundary of $C^o$ this means that $\max_{x \in C} \langle y, x \rangle = 1$ which means that $\max_{x \in S} \langle \pi^*(y), x \rangle = 1$ and so $\pi^*(y)$ is on the boundary of $S^o$, hence $q(y) = p(\pi^*(y)) = 0$.

If we let $d$ be the degree of the algebraic boundary of $C^o$ and $m=\rankpsd(C)$ we have thus shown that $d \leq 2^{m^2}$ which implies $\rankpsd(C) = m \geq \sqrt{\log d}$.
\end{proof}

\paragraph{Application: number of vertices of spectrahedral shadows.} In this \modif{subsection} we discuss an application of Theorem \ref{thm:mainlbpsdrank} to bound the number of \emph{vertices} of spectrahedral shadows. If $C \subseteq \RR^n$ is a convex body and $x \in C$, the \emph{normal cone} of $C$ at $x$ is defined as
\[
N_C(x) := \{ c \in \RR^n : \langle c, z \rangle \leq \langle c,x \rangle \; \forall z \in C \}.
\]
A point $x \in C$ is called a \emph{vertex} if $N_C(x)$ is \emph{full-dimensional}. Observe that any vertex of $C$ must be an extreme point, but not all extreme points are vertices, see Figure \ref{fig:vertex}. Vertices play the role of singularities on the boundaries of convex sets; in fact they are also sometimes called \emph{0-singular points}. It is known, see e.g., \cite[Theorem 2.2.5]{schneiderbook} that any convex set has at most a countable number of vertices.
Vertices of spectrahedra arising from combinatorial optimization problems have been studied in \cite{laurent1995positive,silvatuncelvertices}. The next theorem gives an upper bound on the number of vertices of any spectrahedral shadow. To the best of our knowledge this is the first such bound.
\begin{theorem}
\label{thm:vertices}
If $C$ is a convex body having a semidefinite representation of size $m$, then $C$ has at most $2^{m^2}$ vertices.
\end{theorem}
\begin{proof}
Any vertex of $C$ will contribute a linear factor in the algebraic boundary of $C^o$: indeed if $x$ is a vertex of $C$ then the algebraic boundary of $C$ must contain the hyperplane $\{c \in \RR^n : c^T x = 1\}$ (see e.g., Figure \ref{fig:vertex}(right)). Thus the degree of $\partial_a C^o$ is greater than or equal the number of vertices of $C$. The result follows since the degree of $\partial_a C^o$ is at most $2^{m^2}$.
\end{proof}

\begin{figure}[ht]
  \centering
  \includegraphics[width=12cm]{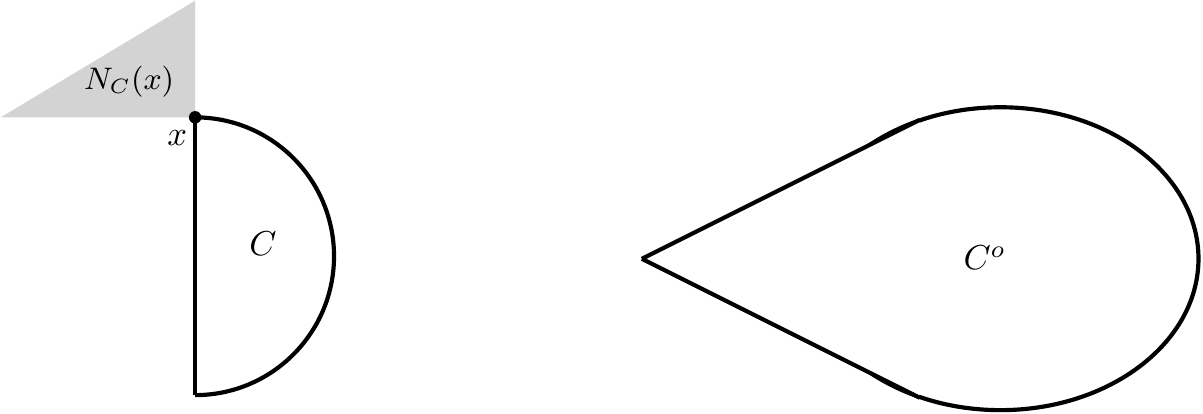}
  \caption{Left: A vertex of a convex set $C$. Right: The polar of $C$. We see that each vertex contributes a hyperplane in the algebraic boundary $\partial_a C^o$.}
  \label{fig:vertex}
\end{figure}

\section{Tightness of lower bound, and algebraic degree of semidefinite programming}

In this section we prove Theorem \ref{thm:tightness}. We will show that the lower bound of Theorem \ref{thm:mainlbpsdrank} is tight up to a constant factor on certain random spectrahedra of appropriate dimension $n$, namely $n \approx t_{m/2}$.

Let $S$ be a spectrahedron defined using matrices $A:=(A_0,\ldots,A_n)$. In the previous section we saw that if we project the following KKT equations 
\begin{equation}
\label{eq:KKTS}
\KKT \; : \;
\begin{cases}
X = A_0 + \cA(x), \quad \cA^*(Z) + c = 0, \quad XZ = 0\\
c^T x = 1
\end{cases}
\end{equation}
on $c \in \CC^n$ we get an algebraic variety
\[
\cV(A) = \elim_{c}({\sf Sols}(\KKT))
\]
that vanishes on the boundary of $S^o$. This variety could coincide exactly with $\partial_a S^o$ but it can also contain spurious components that do not intersect $\partial S^o$ and thus are not in its Zariski closure (see Section \ref{sec:example} later for an example).

In order to prove our result we need to understand the irreducible
components of the variety $\cV(A)$. If we can show that there is an
\emph{irreducible component} $\cW$ of $\cV(A)$ whose intersection with
the boundary of $S^o$ has dimension the one of $\cW$ then we know that
the degree of the algebraic boundary of $S^o$ is at least $\deg
\cW$. When $A$ is generic, the irreducible components of $\cV(A)$ have
been studied in \cite{nie2010algebraic} where it was shown that they
are obtained by imposing rank conditions on the matrices $X$ and $Z$
in the KKT equations, namely by considering the following system for a
fixed $r$:
\begin{equation}
\label{eq:KKTrank0}
\text{KKT}_r \; : \;
\begin{cases}
 X = \cA(x) + A_0, \;\; \cA^*(Z) + c = 0, \;\; XZ = 0\\
 c^T x = 1,\\
 \rank(X) \leq r, \;\; \rank(Z) \leq m-r.
\end{cases}
 \end{equation}
 We think of \eqref{eq:KKTrank0} as a system of equations in
 $(x,X,Z,c)$. If we eliminate the variables $(x,X,Z)$ \modif{from the
   above equations} we get an algebraic variety in $\CC^n$ that is
 contained in $\cV(A)$. We call this variety $\Vr(A)$:
\begin{equation}
\label{eq:Dr*A}
\Vr(A) = \elim_{c}({\sf Sols}(\KKT_r)) \subseteq \cV(A).
\end{equation}
For generic $A$, it was shown in \cite[Theorem 13]{nie2010algebraic}
that $\Vr(A)$ is a hypersurface provided $r$ satisfies the
\emph{Pataki bounds}:
\begin{equation}
\label{eq:pataki}
n \geq t_{m-r} \quad \text{ and } \quad t_r \leq t_m - n.
\end{equation}
Using Bertini theorem one can show that this variety is also
irreducible over $\CC$ provided $n > t_{m-r}$.
\begin{lemma}
\label{lem:irredKKT}
For generic $A_0,\ldots,A_n$ the variety $\Vr(A)$ is
\emph{irreducible} over $\CC$ provided $n > t_{m-r}$.
\end{lemma}
Before proving this lemma we first explain the reason for the
condition $n > t_{m-r}$ (which is stronger than the condition imposed
by the Pataki bound \eqref{eq:pataki}). The variety $\cV_r(A)$ is the
dual of the determinantal variety
$\{x \in \CC^n : \rank(A_0 + x_1 A_1 + \cdots + x_n A_n) \leq
r\}$. The condition $n > t_{m-r}$ rules out the case where this
determinantal variety is zero-dimensional, in which case the dual
variety $\Vr(A)$ is a union of hyperplanes and is thus not
irreducible. Note that if we are only interested in irreducibility
statements over $\QQ$ (assuming that $A_0,\ldots,A_n$ are generic with
entries in $\QQ$) then we do not need to impose such a condition. See
\cite[Remark 2.2]{sinn2015generic} for more on this.
\begin{proof}[Proof of Lemma \ref{lem:irredKKT}]
  The main ingredient of the proof is Bertini's irreducibility theorem
  \modif{\cite[Theorem 4.23]{coursensag}}. We will start by showing
  that the variety
\begin{equation}
\label{eq:KKT7}
X = \cA(x) + A_0, \;\; XZ = 0, \;\; \rank(X) \leq r, \;\; \rank(Z) \leq m-r
\end{equation}
is irreducible for a generic choice of $A_0,\ldots,A_n$. In
\cite[Lemma 6]{nie2010algebraic} it was shown that
$\{XZ=0\}^r := \{(X,Z) : XZ = 0, \rank(X) \leq r, \rank(Z) \leq m-r\}$
is irreducible.  Consider the projection map $u(X,Z) = X$. We know
that $u(\{XZ=0\}^r)$ is the determinantal variety consisting of
symmetric matrices of rank $\leq r$ and has dimension
$t_m-t_{m-r}$. By Bertini theorem \cite[Theorem 4.23]{coursensag} we
know that for a generic affine subspace $L$ of dimension $n$ the
variety $u^{-1}(L)$ is going to be irreducible provided
$t_m-t_{m-r} \geq 1+\codim L = 1+t_m - n$, i.e., provided that
$n \geq t_{m-r}+1$. In other words this tells us that \eqref{eq:KKT7}
is irreducible for a generic choice of $A_0,\ldots,A_n$.

\modif{Consider now the map
  $\phi(x,X,Z) = (x,X,-Z/(\cA^*(Z)^T x),\cA^*(Z)/(\cA^*(Z)^T x))$
  (where the last coordinates stand for $c$). Observe that the image
  of the restriction of $\phi$ to the solution set of \eqref{eq:KKT7}
  is exactly the variety defined by \eqref{eq:KKTrank0}. 
  Since $\phi$ is rational at all points, it is regular \cite[Thm 4, Sec. 3.2]{Shafarevich77}. Because the solution set of \eqref{eq:KKT7} is irreducible, its image by $\phi$ is irreducible. Since $\Vr(A)$ is the projection of an irreducible variety it is also irreducible.}
\end{proof}

The degrees of the irreducible components $\Vr(A)$ were computed (for generic $A=(A_0,\ldots,A_n)$) in \cite{nie2010algebraic,von2009general} and are denoted by $\delta(n,m,r)$. The resulting formulas involve minors of the matrix of binomial coefficients. An elementary analysis of these formulas allows us to show that in a special regime for $n$ and $r$, the algebraic degree is at least $2^{m^2/20}$.

\begin{lemma}
\label{lem:algdegformula}
Assume $m$ even and large enough and consider $n = t_{m/2}+1$ and $r = m/2+1$. 
Then for generic $A=(A_0,\ldots,A_n) \in (\S^m(\CC))^{n+1}$ the variety $\Vr(A)$ has degree $\geq 2^{m^2 / 20}$.
\end{lemma}
\begin{proof}
The proof is in Appendix \ref{app:algdegformula}.
\end{proof}

In order to use Lemma \ref{lem:algdegformula} we need to show that there is at least one choice of $A=(A_0,\ldots,A_n)$ with $n=t_{m/2}+1$ such that the variety $\Vr(A)$ with $r=m/2+1$ will actually belong to $\partial_a S^o$, where $S = \{x : A_0 + x_1 A_1 + \dots +x_nA_n \psd 0\}$.
We can prove this by appealing to results by Amelunxen and B\"urgisser \cite{amelunxen2015intrinsic} where random semidefinite programs were analyzed and where it was shown that every value of rank in the Pataki range occurs with ``positive probability''.

\begin{lemma}
\label{lem:amelunxen}
Let $m$ and $1 \leq n \leq t_m$ be fixed. Let $r$ in the associated Pataki range \eqref{eq:pataki} with the additional constraint $n > t_{m-r}$. Let $\Gamma$ be any Zariski open set in $(\S^m(\CC))^{n+1}$. Then there exists $A=(A_0,\ldots,A_n) \in \Gamma \cap (\S^m(\RR))^{n+1}$ such that the variety $\Vr(A)$ is contained in $\partial_a S^o$ where $S = \{x \in \RR^n : A_0 + x_1 A_1 + \cdots + x_n A_n \psd 0\}$.
\end{lemma}
\begin{proof}
See Appendix \ref{app:amelunxen}.
\end{proof}

The proof of Theorem \ref{thm:tightness} is now complete:

\begin{proof}[Proof of Theorem \ref{thm:tightness}]
Let $m$ be even and large enough and let $n=t_{m/2}+1$. Lemma \ref{lem:amelunxen} with $r=m/2+1$ tells us that there is a spectrahedron such that the variety $\Vr(A)$ is contained in $\partial_a S^o$ where $S = \{x : A_0 + x_1 A_1 + \dots + x_n A_n \psd 0\}$. By Lemma \ref{lem:algdegformula} we know that the degree of this variety is at least $2^{m^2/20}$ and so $d = \deg (\partial_a C^o) \geq 2^{m^2/20}$. But this means that $\rankpsd(S) \leq m \leq \sqrt{20 \log d}$ as desired.
\end{proof}

\section{Example}
\label{sec:example}

In this section we consider an example of spectrahedral shadow to illustrate some of the ideas presented in the proofs of Theorem \ref{thm:mainlbpsdrank} and Theorem \ref{thm:tightness}.

Consider the following linear matrix inequality:
\[
A(x,y,s,t) := \begin{bmatrix}
1+s & t & x+s & y-t\\
t & 1-s & -y-t & x-s\\
x+s & -y-t & 1+x & -y\\
y-t & x-s & -y & 1-x
\end{bmatrix}.
\]
One can show that the projection of the associated spectrahedron on the variables $(x,y)$ is the regular pentagon in the plane, i.e., if we let $S$ be the spectrahedron associated to $A$ and $\pi(x,y,s,t) = (x,y)$ then:
\begin{equation}
\label{eq:Cpentagon}
C := \pi(S) = \conv\left \{ \left(\cos\left(\frac{2k\pi}{5}\right), \sin\left(\frac{2k\pi}{5}\right) \right), k=0,\ldots,4\right\}.
\end{equation}
It is not difficult to see that the polar of $C$ is another regular pentagon but slightly rotated and scaled:
\[
C^o = \frac{1}{\cos(\pi/5)}\conv\left \{ \left(\cos\left(\frac{2(k+1/2)\pi}{5}\right), \sin\left(\frac{2(k+1/2)\pi}{5}\right) \right), k=0,\ldots,4\right\}.
\]
From Section \ref{sec:proof} we know that the KKT equations allow us to get a polynomial that vanishes on the boundary of $C^o$. The associated variety (denoted by $\cV$ in \eqref{eq:defV}) in this case is shown in Figure \ref{fig:pentagon_polar}.
\begin{figure}[ht]
  \centering
  \includegraphics[width=8cm]{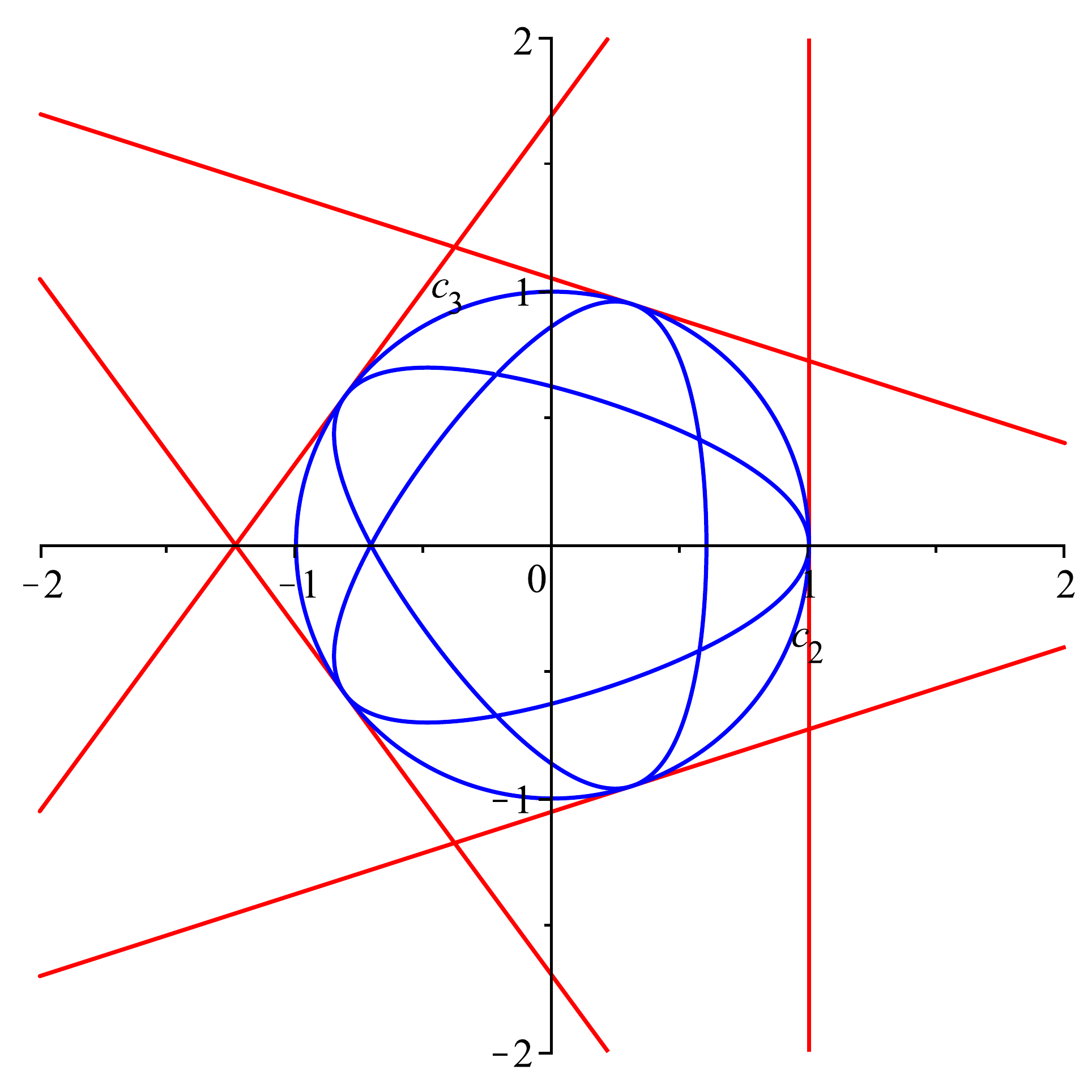}
  \caption{Variety $\cV$ defined in \eqref{eq:defV} that vanishes on the boundary of $C^o$, where $C$ is the regular pentagon, see \eqref{eq:Cpentagon}. We see that $\partial_a C^o \subset \cV$ and that $\cV$ has extra components not in $\partial_a C^o$. These are shown in blue.}
  \label{fig:pentagon_polar}
\end{figure}
We see that the variety $\cV$ contains the algebraic boundary of the polar $C^o$ (red lines). However we also see that it has extra components that are not in $\partial_a C^o$: these extra components are shown in blue in Figure \ref{fig:pentagon_polar}.

\section{Discussion}

The algebraic argument given in this paper can also be used to lower bound the size of second-order cone lifts, or more generally lifts using products of $\S^k_+$. More precisely one can show that if $C$ has a lift using $r$ copies of $\S^k_+$ then $r \geq \frac{1}{k^2} \log d$ where $d$ is the degree of the algebraic boundary of $C^o$. In particular we recover the result of Goemans (Theorem \ref{thm:goemans}) with $k=1$.

There are a couple of questions that we believe it would be interesting to pursue further:
\begin{itemize}
\item \textbf{Polytopes:} One important question is to know whether the lower bound $\rankpsd(C) \geq \sqrt{\log d}$ can be improved in the case where $C$ is a polytope? In particular can the lower bound be improved to $\log d$ in this special case? Recall that if $C$ is a polytope then $d = \deg \partial_a C^o$ is simply its number of vertices.
\item \textbf{Vertices of spectrahedra:} A related question is to know whether the bound of $2^{m^2}$ on the number of vertices of spectrahedral shadows (Theorem \ref{thm:vertices}) is tight? In words, can we find a spectrahedron (or a spectrahedral shadow) that has $2^{\Omega(m^2)}$ vertices? We believe that a natural candidate to try are random spectrahedra of appropriate dimension $n \approx t_{m/2}$. Results in \cite{amelunxen2015intrinsic} can be useful for this question.
\item \textbf{Explicit example:} Thirdly, is there an \emph{explicit} example of a spectrahedron whose polar has an algebraic boundary of degree $2^{\Omega(m^2)}$?
\item \textbf{Analysis of algebraic degree:} Finally we believe it would be useful to have an (asymptotic) analysis of the formulas for the algebraic degrees of semidefinite programming $\delta(n,m,r)$. For this paper we have used elementary manipulations to show that when $n\approx t_{m/2}$ and for certain values of $r$ then $\delta(n,m,r)$ is $2^{\Omega(m^2)}$, but we believe a more complete and systematic study of these quantities can be undertaken. For example we conjecture that in the regime $n \approx t_{m/2}$ the value of $\delta(n,m,r)$ for any $r$ in the Pataki range is $2^{\Omega(m^2)}$. Proving such a result would allow us to simplify the proof of Theorem \ref{thm:tightness} by bypassing the need for Lemma \ref{lem:amelunxen} (it suffices to take any generic spectrahedron of dimension say $n = t_{m/2}+1$ and to observe that \emph{at least one} of the $\Vr(A)$ must belong to $\partial_a S^o$). An analysis of the values of $\delta(n,m,r)$ would also allow us to improve the constants in Theorems \ref{thm:mainlbpsdrank} and \ref{thm:tightness}. For example, where we used the B\'ezout bound in Lemma \ref{lemma:finitenessKKT} one can instead use the quantity $\sum_{r} \delta(n,m,r)$ (where $r$ ranges over the Pataki range) as an upper bound on the number of solutions of the KKT system.
\end{itemize}

\paragraph{Acknowledgments.} We would like to thank Rainer Sinn for
clarifications concerning questions of irreducibility used in Lemma
\ref{lem:irredKKT}, and James Saunderson for comments that helped
improve the paper. We would also like to thank anonymous referees for
their comments on an earlier version of the paper that was submitted
for presentation to the conference MEGA 2017. Finally we would like to
thank Xavier Allamigeon and St\'ephane Gaubert for organizing the
session on Semidefinite Optimization at the PGMO 2016 conference where
this project was started. The second author is supported by the ANR
JCJC GALOP grant.

\appendix

\section{Proof of Lemma \ref{lem:algdegformula}: analysis of the formula for the algebraic degree of semidefinite programming}
\label{app:algdegformula}

In this \modif{subsection} we prove Lemma \ref{lem:algdegformula} which we restate below.

\begin{lemma*}[Restatement of Lemma \ref{lem:algdegformula}]
\modif{Assume $m$ even and large enough and consider $n = t_{m/2}+1$ and $r = m/2+1$. 
Then for generic $A=(A_0,\ldots,A_n) \in (\S^m(\CC))^{n+1}$ the variety $\Vr(A)$ has degree $\geq 2^{m^2 / 20}$.}
\end{lemma*}

For this we rely on the formula for the algebraic degree of semidefinite programming proved in \cite{von2009general}.

Let $\delta(n,m,r)$ be the degree of the variety $\Vr(A)$ where $A$ is a generic pencil $(A_0,\ldots,A_n) \in (\S^m(\RR))^{n+1}$. A formula for $\delta(n,m,r)$ was given in \cite{von2009general} which we describe now. Let $\Psi$ be the (infinite) matrix of binomial coefficients, i.e., $\Psi_{i,j} = \binom{i}{j}$ for $i,j \geq 0$. For $I \subseteq \{1,\dots,m\}$ define $\psi_I$ to be the sum of all the $|I|\times |I|$ minors of $\Psi[I,\cdot]$. For example if $I$ is a singleton we have $\psi_{\{i\}} = 2^{i-1}$.

\begin{theorem}[{\cite{von2009general}, see also \cite{ranestad2012algebraic}}]
For a generic $A=(A_0,\ldots,A_n)$ the algebraic degree of $\Vr(A)$ (see Equation \eqref{eq:Dr*A}) is given by:
\begin{equation}
\label{eq:defdelta}
\delta(n,m,r) = \sum_{\substack{I\subseteq \{1,\dots,m\}\\ |I|=m-r, s(I)=n}} \psi_{I} \psi_{I^c}
\end{equation}
where for $I \subseteq \{1,\dots,m\}$ we denote by $s(I)$ the sum of the elements of $I$, and $I^c = \{1,\dots,m\} \setminus I$.
\end{theorem}

The main purpose of this Appendix is to prove the following lower bound on $\delta(n,m,r)$ in a special regime.

\begin{lemma}
\label{lem:lbalgdeg}
For all large enough even $m$, $n=t_{m/2}+1$ and $r=m/2+1$ we have $\delta(n,m,r) \geq 2^{m^2/20}$.
\end{lemma}
The bounds we give in this appendix are very crude and are not meant to be optimal. We actually conjecture that in the regime $n \approx t_{m/2}$, we have $\delta(t_{m/2},m,r) \geq 2^{\Omega(m^2)}$ for any $r$ in the Pataki range \eqref{eq:pataki}.

In order to prove our result we will first analyze the value of $\psi$ on intervals. We will show
\begin{lemma}
\label{lem:psi-intervals}
For any integers $p \leq q$ we have $\psi_{[p+1,q]} \geq (1+\frac{q-p}{2p-1})^{t_p}$.
\end{lemma}

Before proving Lemma \ref{lem:psi-intervals}, we first see how to use it to prove Lemma \ref{lem:lbalgdeg}.
\begin{proof}[Proof of Lemma \ref{lem:lbalgdeg}]
Consider $I = \{1,\ldots,m/2-2\} \cup \{m\}$. Then $|I| = m/2-1 = m-r$ and $s(I) = t_{m/2}+1 = n$. Thus $\delta(n,m,r) \geq \psi_{I} \psi_{I^c} \geq \psi_{I^c} = \psi_{[m/2-2,m-1]}$. Using Lemma \ref{lem:psi-intervals} we get
\[
\psi_{[m/2-2,m-1]} \geq \left(1 + \frac{m/2+1}{m-3}\right)^{t_{m/2-2}}.
\]
We now use the fact that $1 + \frac{m/2+1}{m-3} \geq 1+1/2 = 3/2$ and $t_{m/2-2} \geq m^2 / 9$ for large enough $m$ to get $\psi_{[m/2-1,m-2]} \geq 2^{(\log_2(3/2)/9) m^2} \geq 2^{m^2/20}$.
\end{proof}

It thus remains to prove Lemma \ref{lem:psi-intervals}. We can get the value of $\psi$ on intervals by considering the case $n=t_{m-r}$ in \eqref{eq:defdelta}. Indeed in this case there is only one set $I$ that satisfies the constraints of the summation \eqref{eq:defdelta} which is $I = \{1,\ldots,m-r\}$. Since $\psi_{[1,m-r]} = 1$ it follows that $\delta(t_{m-r},m,r) = \psi_{[m-r+1,m]}$. On the other hand a simpler formula for $\delta(t_{m-r},m,r)$ was provided in \cite[Corollary 15]{nie2010algebraic}, based on a result by Harris and Tu \cite{harris1984symmetric}. This tells us that
\begin{equation}
\label{eq:harristu}
\delta(t_{m-r},m,r) = \psi_{[m-r+1,m]} = \prod_{i=0}^{m-r-1}\frac{\binom{m+i}{m-r-i}}{\binom{2i+1}{i}}.
\end{equation}
The formula on the right-hand side can be simplified further using simple manipulations to get 
\begin{equation}
\label{eq:Bmrid}
\psi_{[m-r+1,m]} = \prod_{0 \leq i \leq j \leq m-r-1} \frac{r+i+j+1}{i+j+1}.
\end{equation}
To see why this holds, first use the definition of binomial coefficient $\binom{n}{k} = \frac{n\dots (n-k+1)}{k!}$ to get
\begin{equation}
\label{eq:B2rrid1}
\psi_{[m-r+1,m]} = \prod_{i=0}^{m-r-1} \frac{\binom{m+i}{m-r-i}}{\binom{2i+1}{i}} = \prod_{i=0}^{r-1} \frac{(m+i) \dots (r+2i+1)}{(m-r-i)!} \cdot \frac{i!}{(2i+1)\dots (i+2)}
\end{equation}
Separating the terms in \eqref{eq:B2rrid1} we get
\begin{equation}
\label{eq:B2rrid2}
\psi_{[m-r+1,m]} = \left[\prod_{0\leq i \leq j \leq m-r-1} (r+i+j+1)\right] \cdot \left[\prod_{i=0}^{m-r-1} \frac{i!}{(m-r-i)!(2i+1)\dots (i+2)}\right].
\end{equation}
Noting that $\prod_{i=0}^{m-r-1} i! / (m-r-i)! = \frac{1}{(m-r)!} = \prod_{i=0}^{m-r-1} \frac{1}{(i+1)}$ we see that the second factor in \eqref{eq:B2rrid2} is equal to
\begin{equation}
\label{eq:B2rrid3}
\prod_{i=0}^{m-r-1} \frac{1}{(2i+1) \dots (i+2)(i+1)} = \prod_{i=0}^{m-r-1} \prod_{j=0}^{i} \frac{1}{i+j+1}.
\end{equation}
By doing an appropriate change of variables and plugging this back in \eqref{eq:B2rrid2} we get \eqref{eq:Bmrid}.

Now to prove the bound of Lemma \ref{lem:psi-intervals} note that each term in the product \eqref{eq:Bmrid} is at least $1 + \frac{r}{2(m-r)-1}$ and that there are $t_{m-r}$ terms in the product. The statement of Lemma \ref{lem:psi-intervals} corresponds to $p=m-r$ and $q=m$. This completes the proof.

\section{Proof of Lemma \ref{lem:amelunxen}: occurrence of each value of rank in the Pataki range}
\label{app:amelunxen}

In this Appendix we prove Lemma \ref{lem:amelunxen} which we restate here for convenience.
\begin{lemma*}[Restatement of Lemma \ref{lem:amelunxen}]
Let $m$ and $1 \leq n \leq t_m$ be fixed. Let $r$ in the associated Pataki range \eqref{eq:pataki} with the additional constraint $n > t_{m-r}$. Let $\Gamma$ be any Zariski open set in $(\S^m(\CC))^{n+1}$. Then there exists a pencil $A=(A_0,\ldots,A_n) \in \Gamma \cap (\S^m(\RR))^{n+1}$ such that the variety $\Vr(A)$ is contained in $\partial_a S^o$, where $S = \{x \in \RR^n : A_0 + x_1 A_1 + \dots + x_n A_n \psd 0\}$.
\end{lemma*}
\begin{proof}
For convenience in this proof we let, for $A=(A_0,\ldots,A_n) \in (\S^m(\RR))^{n+1}$, $S(A) \subset \RR^n$ denote the associated spectrahedron:
\[
S(A) = \{x \in \RR^n : A_0 + x_1 A_1 + \dots + x_n A_n \psd 0\}.
\]

In the paper \cite[Remark 4.1]{amelunxen2015intrinsic} it is shown that for any $r$ satisfying the Pataki bounds \eqref{eq:pataki} we have
\begin{equation}
\label{eq:patakipositive}
\Pr_{A_0,\ldots,A_n,c} \left[ \rank\left(\argmax_{x \in S(A)} c^T x\right) = r \right] > 0
\end{equation}
where $A_0,\ldots,A_n,c$ are standard Gaussian with respect to the Euclidean inner product. In other words, each value in the Pataki range occurs with positive probability.
Fix $r$ in the Pataki range satisfying $n > t_{m-r}$ and consider
\[
\Omega_r = 
\left\{ (A_0,\ldots,A_n) \in (\S^m(\RR))^{n+1} :  \Pr_{c}\left[\rank\left(\argmax_{x \in S(A)} c^T x\right) = r\right] > 0\right\}.
\]
By \eqref{eq:patakipositive} we know that $\Omega_r$ has positive probability (otherwise the complement of $\Omega_r$ has probability 1 which would contradict \eqref{eq:patakipositive}). Thus this means that $\Omega_r$ must meet $\Gamma$ since $\Gamma$ is Zariski open.

Let $A:=(A_0,\ldots,A_n) \in \Omega_r \cap \Gamma$ and let $S = S(A) = \{x \in \RR^n : A_0 + x_1 A_1 + \dots + x_n A_n \psd 0\}$. To prove our claim we will show that $\Vr(A)$ intersects the boundary $\partial S^o$ along a semialgebraic set of dimension $n-1$. This will prove our claim because if we let $U$ be this semialgebraic set we then have on the one hand $\partial_a S^o \supseteq \bar{U}^Z$ (where $\bar{U}^{Z}$ denotes the Zariski closure) and on the other hand $\bar{U}^{Z} = \Vr(A)$, the latter following from the fact that $\Vr(A)$ is irreducible of dimension $n-1$ and that $\dim_{\CC}(\bar{U}^{Z}) = n-1$ since $U$ is a semialgebraic set of dimension $n-1$, see \cite[Proposition 2.8.2]{bochnak2013real}.

It remains to show that $\Vr(A)$ intersects $\partial S^o$ along a semialgebraic set of dimension $n-1$. To see why this holds let
\[
U = \tilde{U} \cap \partial S^o \quad \text{ where } \quad 
\tilde{U} = \left\{c \in \RR^n : \rank\left(\argmax_{x \in S} c^T x\right) = r \right\}.
\]
By definition of $\Vr(A)$ (recall that $\Vr(A)$ is defined in terms of rank-constrained KKT equations) we have $U \subseteq \Vr(A) \cap \partial S^o$. Now observe that $U$ is a semialgebraic set of dimension $n-1$: indeed note that $\tilde{U}$ has nonempty interior (since it is a semialgebraic set with positive probability, see Lemma \ref{lem:semialgebraicprob}) and so $U = \tilde{U} \cap \partial S^o$ has dimension $n-1$ since for any $\alpha \in \partial S^o$ and neighborhood $A$ of $\alpha$, $\dim(A \cap \partial S^o) = n-1$ (because $\partial S^o$ is the boundary of a full-dimensional convex set).
This completes the proof.
\end{proof}

\begin{lemma}
\label{lem:semialgebraicprob}
If $W \subseteq \RR^N$ is semialgebraic and $W$ has positive probability under the standard Gaussian measure, then $W$ has nonempty interior.
\end{lemma}
\begin{proof}
Any semialgebraic set can be decomposed as a disjoint union of semialgebraic sets that are homeomorphic to $(0,1)^d$ (see \cite[Theorem 2.3.6]{bochnak2013real}). Since $\Pr[W] > 0$, $W$ must have a component that is homeomorphic to $(0,1)^N$ and thus $W$ has nonempty interior.
\end{proof}

\bibliographystyle{alpha}
\bibliography{vertices_spectrahedra}

\end{document}